\documentclass[nocrop]{lms}

\usepackage[T1]{fontenc}
\usepackage[utf8]{inputenc}
\usepackage{amsmath,amssymb}
\usepackage{microtype}
\usepackage{booktabs,tabularx}
\usepackage[hidelinks]{hyperref}
\hypersetup{
  pdftitle={ISP failures and the topology of universal operators},
  pdfauthor={João Carmo, Ben Hur Eidt and Waleed Noor},
  pdfsubject={Invariant subspaces, affine composition operators, and universal operators}
}

\newtheorem{theorem}{Theorem}[section]

\newtheorem{proposition}[theorem]{Proposition}
\newtheorem{corollary}[theorem]{Corollary}
\newcommand{\D}{\mathbb D}
\newcommand{\C}{\mathbb C}
\newcommand{\T}{\mathbb T}
\newcommand{\Mf}{\mathcal M_f}
\newcommand{\norm}[1]{\left\lVert#1\right\rVert}
\newcommand{\clspan}{\overline{\operatorname{span}}}

\simpleequations
\standardsubsections

\title[Universal operators: topology and consequences of ISP failure]
  {Universal operators: topology and consequences of ISP failure}
\author{Jo\~ao Carmo, Ben Hur Eidt \and Waleed Noor}
\classno{47A15 (primary), 47B33, 30H10 (secondary)}

\makeatletter
\def\@footnotetwo{\footnotetext{2020 {\normalfont\itshape
  Mathematics Subject Classification\/} \@classno.}}
\let\ps@lmslast\ps@plain
\makeatother
\journal{}
\begin{document}

\maketitle
\thispagestyle{plain}

\begin{abstract}
Assuming a negative answer to the invariant subspace problem for
separable complex Hilbert spaces, we prove that every universal
operator has a continuum of minimal infinite-dimensional invariant
subspaces with pairwise trivial intersections. These subspaces can
be chosen so that the restrictions are invertible and their powers
decay at any prescribed geometric rate. For an affine composition
operator on the Hardy space \(H^2(\D)\), we derive restrictions on
the zeros, inner factors and Hardy regularity of generators, and
construct minimal models with radial vanishing of any prescribed finite
order. Independently of the invariant subspace problem, we identify
the norm interior, closure and boundary of the universal operator
class. Its interior is the Caradus--Pozzi class, its closure is the
complement of the upper semi-Fredholm class, and its boundary
consists of the operators that are neither upper nor lower
semi-Fredholm.
\end{abstract}

\section{Introduction}
\label{sec:introduction}

The invariant subspace problem (ISP) for Hilbert spaces asks whether every bounded operator on a separable infinite-dimensional complex Hilbert space has a proper nonzero closed invariant subspace. Fix such a Hilbert space \(H\), and write \(\mathcal B(H)\) for its bounded linear operators. An operator \(U\) is \emph{universal} in Rota's sense if, for every \(A\in\mathcal B(H)\), there is a scalar \(c\ne0\), a closed \(U\)-invariant subspace \(M\), and a bounded isomorphism \(J:H\to M\) such that
\[
 UJ=cJA.
\]
Write \(\mathcal U\) for the class of universal operators. Rota's classical model is the backward shift of infinite multiplicity \cite{Rota}. We use the realization
\[
 B_\infty(x_0,x_1,\ldots)=(x_1,x_2,\ldots)
 \quad\text{on }\ell^2(\mathbb N_0;H),
 \qquad S_\infty=B_\infty^*.
\]
Here \(\mathbb N_0=\{0,1,2,\ldots\}\). We fix a unitary
identification of \(\ell^2(\mathbb N_0;H)\) with \(H\) and
henceforth regard \(B_\infty\) and \(S_\infty\) as operators on
\(H\). Then \(B_\infty S_\infty=I\), and \(B_\infty\) is
universal with norm one. 

A nonzero closed invariant subspace is \emph{minimal} if it
contains no proper nonzero closed invariant subspace for
the same operator. For any fixed universal operator \(U\),
the ISP has an affirmative answer if and only
if every minimal \(U\)-invariant subspace is one-dimensional.
Indeed, if the ISP fails, universality realizes a counterexample,
up to similarity and a nonzero scalar multiple, as a restriction
\(U|_M\); the subspace \(M\) is then infinite-dimensional and
minimal. Conversely, any such restriction is itself an ISP
counterexample. This equivalence underlies the results here that are conditional on ISP  failure. Universal operators enjoy the following inheritance property: an operator with a universal restriction to a closed
invariant subspace is universal \cite[Remark~1.4]{Pozzi};
see its block formulation in \cite[Proposition~2.1]{ST}.

For the function-theoretic part of this work, let \(\D\) be the unit disk, \(\T=\partial\D\), and \(\sigma\) normalized arclength measure on \(\T\). Write \(H^p=H^p(\D)\) for the Hardy spaces and \(\norm{f}_p\) for their norms, and \(H^\infty\) consists of bounded analytic functions \cite{Duren}. For \(0\ne f\in H^2\), we use the canonical factorization
\(f=BSO\), where \(B\) is a Blaschke product, \(S\) is a singular
inner factor, and \(O\) is outer. We call the inner factor \emph{purely Blaschke} when
\(S\) is constant, and call \(f\) \emph{outer} when its entire
inner factor is constant. Fix \(0<a<1\), and reserve \(T\) throughout the article for
\[
 \phi_a(z)=az+1-a,\qquad
 T=C_{\phi_a}:H^2\longrightarrow H^2,\qquad Tf=f\circ\phi_a.
\]
For \(f\in H^2\), denote the $T$-cyclic subspace it generates by
\[
 \Mf=\clspan\{T^nf:n\ge0\}.
\]
The operator \(T\) is injective, and the symbol of its iterates $T^n$ satisfy
\begin{equation}
 \phi_a^{\circ n}(z)=1-a^n(1-z)\qquad(n\ge0).
 \label{eq:iterate}
\end{equation}
If \(M\) is minimal for \(T\), then \(\Mf=M\) for every \(0\ne f\in M\). An infinite-dimensional minimal space therefore contains no eigenvector. Carmo and Noor proved that \(T-\lambda I\) is universal for \(0<|\lambda|<a^{-1/2}\) \cite[Theorem 15]{CN}. Hence failure of the Hilbert-space ISP would give minimal infinite-dimensional \(T\)-invariant subspaces in \(H^2\), which we sometimes call \emph{minimal models}.

Section~\ref{sec:multiplicity} constructs a continuum of minimal
models inside every universal operator.
Section~\ref{sec:affine} studies generators for the affine
composition operator and constructs selected minimal models. More concretely, for every infinite-dimensional minimal \(T\)-invariant subspace
\(M\subset H^2\), one has \(M\cap zM=\{0\}\), no
generator extends analytically across \(1\), and no disk zero
can be removed by division by \(z-w\) while remaining in \(M\).
Every positive iterate of a generator has a purely Blaschke inner
factor; if a generator has finitely many disk zeros, some
positive iterate is outer. Moreover, each such \(M\) contains
a dense \(G_\delta\) set of generators whose every forward
iterate lies outside every \(H^p\), \(p>2\). For each \(m>0\) one can additionally
choose a minimal infinite-dimensional model \(N_m\) on which
\(T\) is invertible.
In these selected models, every nonzero generator itself
extends holomorphically to \(\{z:\operatorname{Re}z<1\}\), has an absolutely summable
orbit, and satisfies the uniform radial estimate
\[
 |f(r)|\le C_m\norm{f}_2(1-r)^m
 \qquad(f\in N_m,\ 0<r<1),
\]
where \(C_m\) is independent of \(f\).
Section~\ref{sec:topology} determines the norm interior, closure
and boundary of \(\mathcal U\), and compares them with the
Caradus and Caradus--Pozzi classes of operators
\[
\begin{aligned}
 \mathcal C&=\{A\in\mathcal B(H):\operatorname{Ran}A=H,\ \dim\ker A=\infty\},\\
% AUTHOR REVIEW: Pozzi's class has finite range codimension;
% the supplied manuscript currently says infinite codimension below.
 \mathcal C_+&=\{A\in\mathcal B(H):
    \operatorname{Ran}A\text{ is closed},\
    \dim\ker A=\infty,\ \dim\ker A^*<\infty\}.
\end{aligned}
\]
The criteria of Caradus \cite{Caradus} and Pozzi \cite{Pozzi} give \(\mathcal C\subseteq\mathcal C_+\subseteq\mathcal U\). The upper and lower semi-Fredholm classes are, respectively,
\[
\begin{aligned}
 \Phi_+&=\{A\in\mathcal B(H):\operatorname{Ran}A\text{ is closed},\ \dim\ker A<\infty\},\\
 \Phi_-&=\{A\in\mathcal B(H):\operatorname{Ran}A\text{ is closed},\
                             \dim\ker A^*<\infty\}.
\end{aligned}
\]
We summarize the main results of Section~\ref{sec:topology} in the table below. Abbreviate
\[
 \mathcal D=\mathcal B(H)\setminus\Phi_+,\qquad
 \Sigma=\mathcal B(H)\setminus(\Phi_+\cup\Phi_-).
\]
% AUTHOR REVIEW: Sigma is the class of non-semi-Fredholm operators;
% it is a proper subset of the non-Fredholm operators in general.
Thus \(\Sigma\) consists of operators that are neither upper nor
lower semi-Fredholm. 

\begin{table}[htbp]
\centering
\caption{Topological properties of the universal operator classes.}
\label{tab:topology}
\vspace{5pt}
\small
\setlength{\tabcolsep}{6pt}
\renewcommand{\arraystretch}{1.15}
\begin{tabularx}{\textwidth}{@{}Xccc@{}}
\toprule
Property & \(\mathcal C\) & \(\mathcal C_+\) & \(\mathcal U\)\\
\midrule
Open in norm & Yes & Yes & No\\
Closed in norm & No & No & No\\
Norm interior & \(\mathcal C\) & \(\mathcal C_+\) & \(\mathcal C_+\)\\
Norm closure & \(\mathcal D\) & \(\mathcal D\) & \(\mathcal D\)\\
Norm boundary & \(\mathcal D\setminus\mathcal C\) & \(\Sigma\) & \(\Sigma\)\\

\bottomrule
\end{tabularx}
\end{table}

\section{A continuum of ISP counterexamples}
\label{sec:multiplicity}

The results of this section show that the failure of the ISP would generate continuum-many minimal models inside every universal operator. The construction also permits invertible models.

\begin{theorem}
\label{thm:multiplicity}
Let \(U\in\mathcal U\). If the Hilbert-space ISP has a negative answer, then, for every prescribed \(q\in(0,1)\), there exists a continuum of infinite-dimensional minimal \(U\)-invariant subspaces \(N_t\), \(t\in\C\), and a constant \(\kappa\ge1\) such that
\[
 N_s\cap N_t=\{0\}\quad(s\ne t),\qquad
 U|_{N_t}\text{ is invertible},
\]
and
\[
 \norm{(U|_{N_t})^n}\le\kappa q^n
 \qquad(t\in\C,\ n\ge0).
\]
The constant \(\kappa\) is independent of \(t\) and \(n\). In particular, every vector in every \(N_t\) has an absolutely summable forward orbit.
\end{theorem}

\begin{proof}
Choose \(A\in\mathcal B(H)\) with no proper nonzero closed invariant subspace. Choose a real number \(c>\norm{A}\), and put
\[
 B=cI+A.
\]
The operators \(B\) and \(A\) have exactly the same closed
invariant subspaces, so \(B\) also has no proper nonzero closed
invariant subspace. Moreover, the Neumann series gives
\[
 B^{-1}=\frac1c\sum_{n=0}^{\infty}\left(-\frac Ac\right)^n,
 \qquad \norm{B^{-1}}\le\frac1{c-\norm{A}}.
\]
Thus \(B\) is invertible. The scalar summand below will control the otherwise unspecified
scalar supplied by universality. Choose a real number \(L>0\)
sufficiently large that
\[
 \frac{\norm{U}\norm{B}}{L}<q,
\]
and consider
\[
 D=B\oplus B\oplus LI
 \quad\text{on }H\oplus H\oplus H.
\]
Identify this direct sum unitarily with the Hilbert space on which \(U\) acts. By universality, there are a scalar \(\alpha\ne0\), a closed \(U\)-invariant subspace \(F\), and a bounded isomorphism \(J:H\oplus H\oplus H\to F\) such that
\[
 UJ=\alpha JD.
\]
The third summand makes \(\alpha L\) an eigenvalue of \(U\), and therefore
\[
 |\alpha|L\le\norm{U},\qquad
 \norm{\alpha B}\le\frac{\norm{U}\norm{B}}L<q.
\]

For \(t\in\C\), define
\[
 G_t=\{(x,tx,0):x\in H\},\qquad N_t=JG_t.
\]
Each \(G_t\) is closed and \(D\)-invariant, and \(D|_{G_t}\) is similar to \(B\). Hence \(N_t\) is closed and \(U\)-invariant, with
\[
 U|_{N_t}\sim\alpha B.
\]
Since \(\alpha\ne0\), this restriction is infinite-dimensional and minimal, and it is invertible because \((\alpha B)^{-1}=\alpha^{-1}B^{-1}\). Distinct graphs have zero intersection, and injectivity of \(J\) gives \(N_s\cap N_t=\{0\}\) for \(s\ne t\).

The similarity maps \(J_t:H\to N_t\), defined by \(J_t x=J(x,tx,0)\), satisfy
\[
 \norm{J_t}\le\norm{J}\sqrt{1+|t|^2},\qquad
 \norm{J_t^{-1}}\le\frac{\norm{J^{-1}}}{\sqrt{1+|t|^2}}.
\]
Consequently, with \(\kappa=\norm{J}\norm{J^{-1}}\),
\[
 \norm{(U|_{N_t})^n}
 \le\norm{J_t}\norm{J_t^{-1}}\norm{\alpha B}^n
 \le\kappa q^n.
\]
Summing this geometric bound proves absolute orbit summability.
\end{proof}

\section{A composition operator \texorpdfstring{\(H^2\)}{H2}-model}
\label{sec:affine}

\subsection{Shifted minimal invariant subspaces and zeros of generators}
\label{sec:zeros}

For \(c>0\) write
\[
 S_c(z)=\exp\left(-c\frac{1+z}{1-z}\right).
\]
Carmo and Noor proved that no nonzero vector in \(S_cH^2\) can generate a minimal \(T\)-invariant subspace \cite[Theorem 21]{CN}. We use this obstruction below. Likewise, a nonzero vector of an infinite-dimensional minimal subspace cannot extend analytically across \(1\) \cite[Theorem 16]{CEN}.

\begin{theorem}
\label{thm:noshift}
If \(M\) is a minimal \(T\)-invariant subspace, then
\[
 M\cap zM=\{0\}.
\]
\end{theorem}

\begin{proof}
Suppose that \(M\cap zM\ne\{0\}\). Choose
\(0\ne g\in M\cap zM\) and write \(g=zf\) with \(f\in M\).
Then \(f,zf\in M\) and \(f\ne0\). Let \(R\) be multiplication
by \(1-z\). Thus \(Rf=f-zf\in M\). Direct calculation yields the intertwining identity
\begin{equation}
 TR=aRT.
 \label{eq:intertwine}
\end{equation}
For any polynomial \(p\),
\[
 Rp(T)f=p(a^{-1}T)Rf\in M.
\]
The vectors \(p(T)f\) are dense in \(M\) by minimality. Since \(R\) is bounded on \(H^2\) and \(M\) is closed, \(RM\subseteq M\). Hence \(zM=(I-R)M\subseteq M\). Beurling's theorem now gives \(M=\Theta H^2\) for an inner function \(\Theta\). For any \(c>0\), the nonzero function \(\Theta S_c\) belongs to \(M\) and has an atomic singular inner factor at \(1\). This contradicts \cite[Theorem 21]{CN}, since every nonzero vector of the minimal space \(M\) generates \(M\).
\end{proof}

\begin{corollary}
\label{cor:zero}
Suppose \(M\) is minimal, \(0\ne f\in M\), and \(f(w)=0\) for some \(w\in\D\). Then \(f(z)/(z-w)\notin M\).
\end{corollary}

\begin{proof}
The quotient belongs to \(H^2\): division by a single zero amounts to removing a Blaschke factor and multiplying by the bounded function \((1-\overline w z)^{-1}\). If \(h(z)=f(z)/(z-w)\) also belonged to \(M\), then \(zh=f+wh\in M\), contradicting Theorem \ref{thm:noshift}.
\end{proof}

For \(n\ge0\), define
\[
 \Omega_n=\{z\in\C:\phi_a^{\circ n}(z)\in\D\}
 =\{z\in\C:|z-(1-a^{-n})|<a^{-n}\}.
\]
Then \(\Omega_0=\D\), and the disks increase with \(n\).
For \(n\ge1\), they contain
\(\overline{\D}\setminus\{1\}\), lie in
\(\{z:\operatorname{Re}z<1\}\), and have boundaries tangent
at \(1\).

\begin{proposition}
\label{prop:generators}
Let \(M\) be a minimal \(T\)-invariant subspace. For every \(0\ne f\in M\) and \(N\ge1\), the vector \(T^Nf\) generates \(M\), extends holomorphically to \(\Omega_N\), and has a purely Blaschke inner factor. If \(M\) contains a nonzero function with only finitely many zeros in \(\D\), then \(M\) has an outer generator.
\end{proposition}

\begin{proof}
Composition with \(\phi_a^{\circ N}\) extends \(T^Nf\) to \(\Omega_N\). The operator \(T\) is injective, so \(T^Nf\ne0\); minimality gives \(\mathcal M_{T^Nf}=M\). A standard consequence of canonical Hardy-space factorization is that the singular inner measure of a function analytic through \(\T\setminus\{1\}\) can only be supported at \(1\) \cite{Duren}. The atomic-factor obstruction \cite[Theorem 21]{CN} rules out mass there. Thus the singular inner factor is trivial. If \(f\) has finitely many zeros, then their distance from \(1\) is positive. For all sufficiently large \(N\), the shrinking disk \(\phi_a^{\circ N}(\D)\) misses those zeros, so \(T^Nf\) is zero-free in \(\D\). It has no singular inner factor by the first part; hence its inner factor is constant and it is outer, up to a nonzero scalar.
\end{proof}

\subsection{Non-membership in Hardy classes $H^p$, $p>2$}
\label{sec:hardy}

We will need a compactness criterion recorded in
\cite[Result~1.5.1]{Jarchow}: a composition operator
\(C_\phi:H^p\to H^q\), with \(0<q<p\le\infty\),
is compact whenever the boundary contact set 
\[
E_\phi=\{\zeta\in\T:|\phi(\zeta)|=1\}
\] 
has arc-length measure zero. For our affine symbol, we have $E_{\phi_a}=\{1\}$ for all \(0<a<1\). Consequently, \(T:H^p\to H^2\) is compact for every \(p>2\). 
\begin{theorem}
\label{thm:generic}
Suppose \(M\) is infinite-dimensional and minimal. There is a dense \(G_\delta\) subset \(\mathcal R\subset M\) such that every \(f\in\mathcal R\) satisfies
\[
 T^nf\notin H^p
 \quad\text{for every }n\ge0\text{ and every }p>2.
\]
In particular, functions bounded on some tangent disk \(\phi_a^{\circ N}(\D)\) form a meagre subset of \(M\).
\end{theorem}

\begin{proof}
Fix \(p>2\), \(n\ge0\), and \(k\ge1\), and put
\[
 E_{p,n,k}=\{f\in M:T^nf\in H^p,\
                          \norm{T^nf}_p\le k\}.
\]
This set is closed in the \(H^2\) topology. Indeed, \(H^2\) convergence implies local uniform convergence; after applying \(T^n\), Fatou's lemma on each concentric circle preserves the \(H^p\) bound. If \(E_{p,n,k}\) had a nonempty interior, subtraction and scaling would give a constant \(C\) such that
\begin{equation}
 \norm{T^nf}_p\le C\norm{f}_2
 \qquad(f\in M).
 \label{eq:hpbound}
\end{equation}
The compactness recalled above would make
\(T^{n+1}|_M=T\circ(T^n|_M)\) compact. Put \(A=T|_M\). The compact operator \(A^{n+1}\) is nonzero
because \(T\) is injective, and it commutes with \(A\).
Moreover, \(A\) is not a scalar operator: otherwise every
closed subspace of the infinite-dimensional space \(M\)
would be invariant. Lomonosov's theorem therefore yields
a proper nonzero closed \(A\)-invariant subspace,
contradicting minimality \cite{Lomonosov}. Thus each \(E_{p,n,k}\) is nowhere dense. Take a countable sequence \(p_j>2\) decreasing to \(2\). Baire's theorem shows that the complement of \(\bigcup_{j,n,k}E_{p_j,n,k}\) is a dense \(G_\delta\). If a vector or any of its iterates belonged to some \(H^p\) with \(p>2\), it would belong to \(H^{p_j}\) for a suitable \(p_j<p\), contrary to its choice. Finally, boundedness of \(f\) on \(\phi_a^{\circ N}(\D)\) means \(T^Nf\in H^\infty\subset H^{p_j}\), so such vectors lie in the meagre union already identified.
\end{proof}

For \(f\in\mathcal R\), the function \(Tf\) extends analytically across \(\T\setminus\{1\}\) by Proposition \ref{prop:generators}, but \(Tf\notin H^p\) for any \(p>2\). Thus this failure of higher integrability is associated with behavior near the boundary point \(1\). The conclusion is generic: it does not determine whether a
minimal subspace can contain an exceptional nonzero vector
in some \(H^p\), \(p>2\).

\subsection{Minimal models with radial vanishing}
\label{sec:radial}

Assuming that the ISP fails, we construct minimal invariant
subspaces whose vectors exhibit uniform radial decay at \(1\)
of any prescribed finite order. These spaces still satisfy Theorem~\ref{thm:generic}, so
rapid radial decay at $1$ can coexist with generic
non-membership in every \(H^p\), \(p>2\).

\begin{theorem}
\label{thm:radial}
Assume that \(T\) has an infinite-dimensional minimal invariant subspace. For every \(m>0\) there is an infinite-dimensional minimal invariant subspace \(N_m\) and a constant \(C_m\) such that
\begin{equation}
 |f(r)|\le C_m\norm{f}_2(1-r)^m
 \qquad(f\in N_m,\ 0<r<1).
 \label{eq:radial}
\end{equation}
Moreover, the operator series
\(\sum_{n\ge0}(T|_{N_m})^n\) converges absolutely in operator
norm. In particular, every vector in \(N_m\) has an
absolutely summable \(H^2\) orbit.
\end{theorem}

\begin{proof}
Let \(M_0\) be an infinite-dimensional minimal subspace and set \(A=T|_{M_0}\). The operator \(A\) has no proper nonzero closed invariant subspace. Fix \(m>0\) and put \(q=a^m<1\). Choose a real \(\lambda\) with \(0<\lambda<q\); it also satisfies \(\lambda<a^{-1/2}\). On the Hilbert direct sum \(K=\bigoplus_{j\ge1}M_0\), consider the bounded operator
\[
 B=\bigoplus_{j\ge1}\frac Aj.
\]
Identify \(K\) unitarily with \(H^2\) to apply the affine universality result from the introduction. It provides a nonzero scalar \(\alpha\), a closed invariant subspace \(F\subseteq H^2\), and a bounded isomorphism \(J:K\to F\) such that
\begin{equation}
 (T-\lambda I)J=\alpha JB.
 \label{eq:model}
\end{equation}
Let \(K_j\) denote the \(j\)th coordinate copy of \(M_0\) and put \(N_j=JK_j\). Each \(N_j\) is closed and \(T\)-invariant, and \(T|_{N_j}\) is similar to \(\lambda I+\alpha A/j\). The latter operator has exactly the invariant subspaces of \(A\); thus every \(N_j\) is infinite-dimensional and minimal. Choose \(j\) large enough that \(\lambda+|\alpha|\norm{A}/j<q\), and set \(N_m=N_j\). Writing \(\kappa=\norm{J}\norm{J^{-1}}\), equation \eqref{eq:model} gives
\begin{equation}
 \norm{T^n|_{N_m}}\le\kappa q^n=\kappa a^{nm}
 \qquad(n\ge0).
 \label{eq:powerdecay}
\end{equation}
The last assertion of the theorem follows immediately by summing this geometric bound. For the radial estimate, given \(0<r<1\), choose \(n\ge0\) with
\(a^{n+1}<1-r\le a^n\), and define
\[
 u=1-\frac{1-r}{a^n}\in[0,1-a].
\]
Then \(r=\phi_a^{\circ n}(u)\). The \(H^2\) point-evaluation bound on the fixed interval \([0,1-a]\) and \eqref{eq:powerdecay} yield
\[
 |f(r)|=|(T^nf)(u)|
 \le\frac{\kappa}{\sqrt{1-(1-a)^2}}\,a^{nm}\norm{f}_2.
\]
Since \(a^{n+1}<1-r\), we have \(a^{nm}<a^{-m}(1-r)^m\). This proves \eqref{eq:radial}, with a constant independent of \(f\) and \(r\).
\end{proof}

\begin{corollary}
\label{cor:radialinvertible}
Under the hypothesis of Theorem~\ref{thm:radial},
\(T|_{N_j}\) is invertible for every sufficiently large \(j\),
where \((N_j)\) is the family constructed in its proof.
\end{corollary}

\begin{proof}
Use the notation of Theorem \ref{thm:radial}: \(q=a^m\), \(0<\lambda<q\), \(A=T|_{M_0}\), and
\[
 T|_{N_j}\sim B_j:=\lambda I+\frac{\alpha}{j}A.
\]
Choose \(j\) large enough that
\[
 \eta_j:=\frac{|\alpha|\norm{A}}j
 <\min\{\lambda,q-\lambda\}.
\]
This is possible because \(\eta_j\to0\). The inequality \(\lambda+\eta_j<q\) is precisely the condition required for the forward decay and radial estimates in the theorem. The additional inequality \(\eta_j<\lambda\) makes \(B_j\) invertible by the Neumann series:
\[
 B_j^{-1}
 =\frac1\lambda\sum_{k=0}^{\infty}
 \left(-\frac{\alpha A}{j\lambda}\right)^k,
 \qquad
 \norm{B_j^{-1}}\le\frac1{\lambda-\eta_j}.
\]
Similarity transfers invertibility to \(T|_{N_j}\). Setting \(N_m=N_j\) preserves all the conclusions of Theorem \ref{thm:radial}. With \(\kappa=\norm{J}\norm{J^{-1}}\), it also gives
\[
 \norm{(T|_{N_m})^{-n}}
 \le\kappa(\lambda-\eta_j)^{-n}
 \qquad(n\ge0).
\]
Thus invertibility is available within the original concrete family, simply by choosing \(j\) sufficiently large.
\end{proof}
Futhermore, vectors in an invertible minimal model extend analytically to a half-plane.
\begin{proposition}
\label{prop:halfplane}
Let \(M\subseteq H^2\) be a closed \(T\)-invariant subspace, and suppose that \(A=T|_M\) is invertible. Every \(f\in M\) has a unique holomorphic extension to
\[
 \mathcal H=\{z\in\C:\operatorname{Re}z<1\}
\]
which we continue to denote by \(f\). For every compact \(K\subset\mathcal H\), there is a constant \(C_K\) such that
\begin{equation}
 \sup_{z\in K}|f(z)|\le C_K\norm{f}_2
 \qquad(f\in M).
 \label{eq:halfplanebound}
\end{equation}
\end{proposition}

\begin{proof}
For \(n\ge0\), put \(g_n=A^{-n}f\), and use the disks
\(\Omega_n\) defined in Subsection~\ref{sec:zeros}. On \(\Omega_n\), define
\[
 F_n(z)=g_n\bigl(\phi_a^{\circ n}(z)\bigr).
\]
On \(\D\), one has \(F_n=T^ng_n=f\). The disks are increasing, so the identity theorem shows that these extensions agree on their overlaps. To identify their union, write \(w=1-z\). Then
\[
 z\in\Omega_n
 \quad\Longleftrightarrow\quad
 |w|^2<2a^{-n}\operatorname{Re}w.
\]
This holds for some \(n\) precisely when \(\operatorname{Re}w>0\). Hence \(\bigcup_{n\ge0}\Omega_n=\mathcal H\), and the agreed extensions define a holomorphic function there. Uniqueness follows from agreement on \(\D\). For a nonempty compact \(K\subset\mathcal H\), choose \(n\) with \(K\subset\Omega_n\) and put \(r_K=\max_{z\in K}|\phi_a^{\circ n}(z)|<1\). The Hardy-space point-evaluation estimate gives
\[
 \sup_{z\in K}|F_n(z)|
 \le\frac{\norm{A^{-n}}}{\sqrt{1-r_K^2}}\,\norm{f}_2.
\]
This proves \eqref{eq:halfplanebound}, with a constant independent of \(f\).
\end{proof}

Proposition \ref{prop:halfplane} applies to the invertible models chosen in Corollary \ref{cor:radialinvertible}, so their vectors extend holomorphically to \(\mathcal H\). 

\begin{corollary}
\label{cor:summable}
The following statements are equivalent:
\begin{flushenumerate}[iv]
 \item the Hilbert-space ISP has an affirmative answer;
 \item whenever \(f\in H^2\) has \(\sum_{n\ge0}\norm{T^nf}_2<\infty\) and \(\Mf\) is minimal, \(\Mf\) is one-dimensional;
 \item whenever \(\sum_{n\ge0}T^nf\) converges in \(H^2\) and \(\Mf\) is minimal, \(\Mf\) is one-dimensional;
 \item for every \(0\ne f\in H^2\) satisfying \(\lim_{r\uparrow1}f(r)=0\), \(\Mf\) is minimal if and only if it is one-dimensional.
\end{flushenumerate}
\end{corollary}

\begin{proof}
Under (i), every minimal invariant subspace of \(T\) is one-dimensional \cite[Corollary 16]{CN}, so (ii) and (iii) follow. Since a one-dimensional invariant subspace is always minimal, (iv) follows as well. Statement (iii) implies (ii). If (i) fails, universality produces an infinite-dimensional minimal subspace for \(T\). Theorem \ref{thm:radial} then gives a minimal space \(N_m\), with \(m>0\), on which every orbit is absolutely summable and every vector has radial limit zero at \(1\). For every \(0\ne f\in N_m\), minimality gives \(\Mf=N_m\); hence \(f\) violates (ii), (iii), and (iv). Thus each of these statements implies (i).
\end{proof}

The corollary clarifies two open problems arising in
\cite{CEN,CN}. First, part~(iii) shows that the conjecture
in \cite[Section~3]{CEN} is equivalent to
an affirmative solution of the ISP. Second, part~(iv) reduces the full ISP to the radial-zero
case. Following \cite{CN}, distinguish generators whose
radial limit at \(1\) is finite and nonzero (A), is zero
(B), or does not exist (C). Any failure of the ISP would already produce
a minimal infinite-dimensional model whose every nonzero
generator belongs to case (B).
Thus cases (A) and (C) need not be treated separately.

\section{The topology of universal operators}
\label{sec:topology}

Recall the classes \(\Phi_+\), \(\Phi_-\), \(\mathcal C\),
\(\mathcal C_+\) and \(\mathcal U\) from
Section~\ref{sec:introduction}. Schroderus and Tylli observed
the norm openness of \(\mathcal C_+\) and showed that
\(\mathcal U\) is not norm open
\cite[Section 2 and Example 2.2]{ST}.
The closure formula for \(\mathcal C\) follows from
Niemiec's description of two-sided invertible orbits
\cite[Corollary 3.12(iii)(b)]{Niemiec}.
The common boundary of the semi-Fredholm index classes is
classical; see Mbekhta \cite{Mbekhta} and Niemiec
\cite[Theorem 5.2]{Niemiec}. Combining this topology with
the criteria of Caradus and Pozzi gives the following topological
description of the full universal class.  

\subsection{Norm closure, interior, and boundary}

The semi-Fredholm classes \(\Phi_+\) and \(\Phi_-\) are
norm open \cite[Theorem 5.2]{Niemiec}. The norm openness
of \(\mathcal C_+\) is also recorded in
\cite[Section 2]{ST}.

\begin{theorem}
\label{thm:topology}
With all closures, interiors, and boundaries taken in the operator norm,
\begin{align}
 \overline{\mathcal U}=\overline{\mathcal C_+}
   =\overline{\mathcal C}
   &=\mathcal B(H)\setminus\Phi_+,\label{eq:universalclosure}\\
 \operatorname{int}\mathcal U
   &=\mathcal C_+=\Phi_-\setminus\Phi_+,\label{eq:universalinterior}\\
 \partial\mathcal U
   &=\mathcal B(H)\setminus(\Phi_+\cup\Phi_-).
   \label{eq:universalboundary}
\end{align}
Consequently, \(\mathcal U\) is neither closed nor open, has nonempty interior, and is not norm dense. Its boundary is closed and nowhere dense.
\end{theorem}

\begin{proof}
Every universal operator necessarily has infinite-dimensional kernel. Since \(\Phi_+\) is norm open and disjoint from \(\mathcal U\),
\[
 \overline{\mathcal U}\subseteq\mathcal B(H)\setminus\Phi_+.
\]
For the reverse inclusion, let \(A\notin\Phi_+\) and fix \(\delta>0\). Write \(|A|=(A^*A)^{1/2}\), and let \(\mathbf1_E(X)\)
denote the spectral projection of a positive operator \(X\)
associated with a Borel set \(E\). Take the orthogonal
spectral projections
\[
 P=\mathbf1_{[0,\delta)}(|A|),\qquad
 Q=\mathbf1_{[0,\delta)}(|A^*|).
\]
The space \(PH\) is infinite-dimensional. Indeed, if it were finite-dimensional, then \(\ker A\subseteq PH\) would be finite-dimensional, and \(|A|\) would be bounded below on \((I-P)H\). The finite-dimensional part also has a positive lower bound after its kernel is removed, so \(A\) would have closed range and belong to \(\Phi_+\). Write \(A=V|A|\) for its polar decomposition. Spectral calculus and \(V\) show that
\[
 A(PH)\subseteq QH,\qquad \norm{AP}\le\delta,
 \qquad
 A:(I-P)H\longrightarrow(I-Q)H
 \text{ is an isomorphism}.
\]
Indeed, \(AP=QA\), and \(V\) maps \((I-P)H\) unitarily
onto \((I-Q)H\). On \((I-P)H\), the positive operator
\(|A|\) is bounded below by \(\delta\), so the inverse of
the displayed restriction has norm at most \(1/\delta\). Since \(PH\) is infinite-dimensional and separable, there is a surjection
\[
 C:PH\longrightarrow QH,\qquad
 \dim\ker C=\infty,\qquad \norm{C}\le1.
\]
If \(QH\) is infinite-dimensional, choose a coisometry with infinite-dimensional kernel; if it is finite-dimensional, choose a projection onto finitely many coordinates followed by an isometry. When \(QH=\{0\}\), take \(C=0\). Set
\[
 A_\delta=A(I-P)+\delta CP.
\]
Relative to the domain decomposition \(PH\oplus(I-P)H\) and the range decomposition \(QH\oplus(I-Q)H\), this is the direct sum of \(\delta C\) and an isomorphism. It is therefore surjective with infinite-dimensional kernel, so \(A_\delta\in\mathcal C\) by Caradus's criterion. Moreover,
\[
 \norm{A_\delta-A}=\norm{\delta CP-AP}\le2\delta.
\]
Letting \(\delta\downarrow0\) and using \(\mathcal C\subseteq\mathcal C_+\subseteq\mathcal U\) proves \eqref{eq:universalclosure}. For the interior, Pozzi's criterion and the norm openness
of \(\mathcal C_+\) give
\(\mathcal C_+\subseteq\operatorname{int}\mathcal U\). Conversely, suppose \(U\in\mathcal U\setminus\mathcal C_+\). It is neither upper nor lower semi-Fredholm. Hence both \(PH\) and \(QH\) above are infinite-dimensional for every \(\delta>0\); for \(QH\), apply the preceding argument to \(U^*\), since \(U\notin\Phi_-\) if and only if \(U^*\notin\Phi_+\). Choose a unitary \(W:PH\to QH\) and replace \(C\) by \(W\). Then
\[
 V_\delta=U(I-P)+\delta WP
\]
is invertible and satisfies \(\norm{V_\delta-U}\le2\delta\). An invertible operator cannot be universal, because its restrictions have trivial kernel. Thus \(U\) is not an interior point, proving \eqref{eq:universalinterior}. Subtracting the interior from the closure gives \eqref{eq:universalboundary}. This boundary is closed. Every operator on it has the unitary approximation just constructed, so no open ball can lie in the boundary; hence it is nowhere dense. Finally, if \(U\) is universal, then \(U/n\) is universal and tends in norm to \(0\), which is not universal. Thus \(\mathcal U\) is not closed. It is not open because \(B_\infty\oplus0\) is universal by the inheritance principle, but has infinite-codimensional range and therefore lies outside \(\mathcal C_+\). Nonempty interior follows from \(B_\infty\in\mathcal C\); failure of density follows from the open ball \(\norm{A-I}<1\), which consists of invertible operators and lies in \(\Phi_+\).
\end{proof}

Neither \(\mathcal C\) nor \(\mathcal C_+\) is norm closed:
the operators \(B_\infty/n\) belong to both classes and
converge to \(0\), which belongs to neither. Surjectivity is norm open, so \(\mathcal C=\mathcal C_+\cap\{A:\operatorname{Ran}A=H\}\) is open as well. Thus its norm interior is \(\mathcal C\), and its boundary is \(\mathcal D\setminus\mathcal C\). Together with Theorem \ref{thm:topology}, this proves the norm entries of Table \ref{tab:topology}. Formula \eqref{eq:universalinterior} gives a precise interpretation of Pozzi's criterion: the universal operators whose universality survives every sufficiently small norm perturbation are exactly the Caradus--Pozzi operators. The full universal class differs from this open class only by a subset of the closed nowhere dense set of non-semi-Fredholm operators. Rota's backward shift model introduced in Section \ref{sec:introduction} makes the nonempty interior quantitative.

\begin{corollary}
\label{cor:universalball}
If \(\norm{V-B_\infty}<1\), then \(V\in\mathcal C\). Consequently the distance from \(B_\infty\) to the complement of \(\mathcal U\) is exactly \(1\).
\end{corollary}

\begin{proof}
The operator \(VS_\infty=I+(V-B_\infty)S_\infty\) is invertible by the Neumann series, so \(V\) is surjective. The map
\[
 x\longmapsto x-S_\infty(VS_\infty)^{-1}Vx
 \qquad(x\in\ker B_\infty)
\]
takes values in \(\ker V\) and is injective: its second term belongs to \(\operatorname{Ran}S_\infty\), which is orthogonal to \(\ker B_\infty\). Thus \(\ker V\) is infinite-dimensional. The distance is at least \(1\) by this ball inclusion and at most \(1\) because \(0\notin\mathcal U\).
\end{proof}

Curiously, small scalar translations preserve universality, even at the boundary points of \(\mathcal U\).

\begin{corollary}
\label{cor:universaltranslations}
For every \(U\in\mathcal U\) there exists an \(r>0\) such that \(U-\lambda I\) is universal whenever \(|\lambda|<r\).
\end{corollary}

\begin{proof}
Choose \(M\) with \(U|_M\sim cB_\infty\), \(c\ne0\), and put \(r=|c|\). For \(|\lambda|<r\), write \(\mu=\lambda/c\). Since
\[
 (B_\infty-\mu I)S_\infty=I-\mu S_\infty,
\]
the Neumann series makes \(B_\infty-\mu I\) surjective. Its kernel contains the vectors
\[
 (v,\mu v,\mu^2v,\ldots),\qquad v\in H,
\]
so the kernel is infinite-dimensional. Therefore
\((U-\lambda I)|_M\sim c(B_\infty-\mu I)\) is universal by Caradus, and the inheritance principle applies proving the universality of \(U-\lambda I\).
\end{proof}

 There are two further consequences of Theorem \ref{thm:topology}. First, every compact operator belongs to \(\overline{\mathcal U}\), since a compact operator on infinite-dimensional \(H\) cannot be upper semi-Fredholm. Also, compact operators are not universal, so they must belong to $\partial\mathcal U$.  Second, \(\mathcal U\) has the Baire property as it is the union of an open set and a meagre set.

\section*{Declaration of AI use}

The authors used GPT-6.1 Sol throughout the preparation of this
manuscript. They had been investigating related questions for
several years before using this tool. Nevertheless, AI played
a crucial role in the development of the present paper,
including the formulation of results, the development and
revision of proofs, the identification and checking of
references, and the organisation and editing of the exposition.
The collaboration involved more than fifty author prompts in an
extended, iterative exchange. The authors take full
responsibility for the mathematical claims, citations, and
final text.

\section{Acknowledgments} Ben-Hur Eidt is supported by the São Paulo Research Foundation (FAPESP), Brasil. Process Number 2024/14727-4.

\affiliationone{Jo\~ao Carmo\\
  Instituto Federal da Bahia (IFBA)\\
  Campus Seabra\\
  Seabra, Bahia, Brasil
  \email{joaoribeiro@ifba.edu.br}}%
\affiliationtwo{Ben Hur Eidt\\
  ICMC, Universidade de S\~ao Paulo\\
  S\~ao Carlos, SP, Brazil
  \email{benhur@icmc.usp.br}}%
\hspace*{-10pt}% Remove unused trailing padding.
\affiliationthree{Waleed Noor\\
  IMECC, Universidade Estadual\\
  de Campinas\\
  Campinas, SP, Brazil
  \email{waleed@unicamp.br}}
\end{document}